\documentclass[10pt]{amsart}

\usepackage{amsmath,amsfonts}
\usepackage[english]{babel}

\newtheorem{theorem}{Theorem}[section]
\newtheorem{lemma}[theorem]{Lemma}

\newtheorem{proposition}[theorem]{Proposition}
\newtheorem{corollary}[theorem]{Corollary}

\newtheorem{question}[theorem]{Question}

\newtheorem{example}[theorem]{Example}

\begin{document}

\title{On $\alpha$-embedded sets and extension of mappings}

\author{Olena Karlova}


\address{ Chernivtsi National University,
                    Department of Mathematical Analysis,
                    Kotsjubyns'koho 2,
                    Chernivtsi 58012,
                    UKRAINE}

\begin{abstract}
We introduce and study $\alpha$-embedded sets and apply them to generalize the Kuratowski Extension Theorem.

\end{abstract}

\maketitle

\renewcommand{\theenumi}{\roman{enumi}}

\section{Introduction}

A subset $A$ of a topological space $X$ is called {\it functionally open} ({\it functionally closed}) if there exists a continuous function $f:X\to [0,1]$ such that $A=f^{-1}((0,1])$ ($A=f^{-1}(0)$).

Let ${\mathcal G}_0^*(X)$ and ${\mathcal
F}_0^*(X)$ be the collections of all functionally open and functionally closed subsets of a topological space $X$, respectively. Assume that the classes ${\mathcal G}_\xi^*(X)$ and ${\mathcal
F}_\xi^*(X)$ are defined for all $\xi<\alpha$, where
$0<\alpha<\omega_1$. Then, if $\alpha$ is odd, the class ${\mathcal G}_\alpha^*(X)$ (${\mathcal F}_\alpha^*(X)$)  consists of all countable intersections  (unions)
of sets of lower classes, and, if $\alpha$ is even, the class ${\mathcal G}_\alpha^*(X)$ (${\mathcal F}_\alpha^*(X)$)  consists of all countable unions
(intersections) of sets of lower classes. The classes  ${\mathcal F}_\alpha^*(X)$ for odd $\alpha$ and  ${\mathcal G}_\alpha^*(X)$ for even  $\alpha$ are said to be
{\it functionally additive}, and the classes ${\mathcal F}_\alpha^*(X)$
for even $\alpha$ and ${\mathcal G}_\alpha^*(X)$ for odd
$\alpha$ are called {\it functionally multiplicative}. If a set belongs to the $\alpha$'th functionally additive and to the $\alpha$'th functionally multiplicative class simultaneously, then it is called {\it functionally ambiguous of the $\alpha$'th class}. For every $0\le\alpha<\omega_1$ let
$$
{\mathcal B}_\alpha^*(X)={\mathcal F}_\alpha^*(X)\cup {\mathcal G}_\alpha^*(X)
$$
 and let
$$
{\mathcal B}^*(X)=\bigcup\limits_{0\le\alpha<\omega_1} {\mathcal B}_\alpha^*(X).
$$
If $A\in {\mathcal B}^*(X)$, then $A$ is said to be {\it a functionally measurable} set.

If $P$ is a property of mappings, then by $P(X,Y)$ we denote the collection of all mappings $f:X\to Y$ with the property $P$.  Let $P(X)$ ($P^*(X)$) be the collection of all real-valued (bounded) mappings on $X$ with a property $P$.

By the symbol $C$ we denote, as usually, the property of continuity.

Let $K_0(X,Y)=C(X,Y)$. For an ordinal $0<\alpha<\omega_1$ we say that a mapping \mbox{$f:X\to Y$}  {\it  belongs to the $\alpha$'th functional Lebesgue class}, $f\in K_\alpha(X,Y)$, if the preimage $f^{-1}(V)$ of an arbitrary open set $V\subseteq Y$ is of the $\alpha$'th functionally additive class in $X$.

A subspace $E$ of $X$ is {\it $P$-embedded} ({\it $P^*$-embedded}) in $X$ if every (bounded) function $f\in P(E)$ can be extended to a (bounded)  function $g\in P(X)$.

A subset $E$ of $X$ is said to be {\it $z$-embedded} in $X$ if every functionally closed set in $E$ is the restriction of a functionally closed set in $X$ to $E$. It is well-known that
\begin{center}
$E$ -- $C$-embedded $\Rightarrow$ $E$ -- $C^*$-embedded $\Rightarrow$ $E$ -- $z$-embedded.
\end{center}

Recall that sets $A$ and $B$ are {\it completely separated in $X$} if there exists a continuous function $f:X\to [0,1]$ such that $A\subseteq f^{-1}(0)$ and $B\subseteq f^{-1}(1)$.

The following theorem was proved in \cite[Corollary 3.6]{BlairHager}.

\begin{theorem}[Blair-Hager] A subset $E$ of a topological space $X$ is $C$-embedded in $X$ if and only if $E$ is $z$-embedded in $X$ and $E$ is completely separated from every functionally closed set in $X$ disjoint from $E$.
\end{theorem}

It is natural to consider $P$- and $P^*$-embedded sets if $P=K_\alpha$ for $\alpha>0$. In connection with this we introduce and study a class of $\alpha$-embedded sets which coincides with the class of $z$-embedded sets when $\alpha=0$.  In Section~3 we generalize the notion of completely separated sets to $\alpha$-separated sets. Section~4 deals with ambiguously $\alpha$-embedded sets which play the important role in the extension of bounded $K_\alpha$-functions. In the fifth section we prove an analog of the Tietze-Uryhson Extension Theorem for $K_\alpha$-functions. Section 6 concerns the question when $K_1$-embedded sets coincide with $K_1^*$-embedded sets. The seventh section presents a generalization of the Kuratowski Theorem \cite[p.~445]{Ku1} on extension of $K_\alpha$-mappings with values in Polish spaces.

\section{$\alpha$-embedded sets}

Let $0\le\alpha<\omega_1$. A subset $E$ of a topological space $X$ is  {\it $\alpha$-embedded} in $X$ if for any set $A$ of the $\alpha$'th functionally additive  (multiplicative) class in $E$ there is a set $B$ of the $\alpha$'th functionally additive (multiplicative) class in  $X$ such that $A=B\cap E$.

\begin{proposition}\label{alpha_in_alpha}
  Let  $X$ be a topological space, $0\le\alpha<\omega_1$  and let $E\subseteq X$ be an $\alpha$-embedded set of the $\alpha$'th functionally additive (multiplicative) class in $X$. Then every set of the $\alpha$'th functionally additive (multiplicative) class in $E$ belongs to the $\alpha$'th functionally additive (multiplicative) class in  $X$.
\end{proposition}

\begin{proof}
For a set $C$ of the $\alpha$'th functionally additive (multiplicative) class in $E$ we choose a set $B$ of the $\alpha$'th functionally additive (multiplicative) class in $X$ such that $C=B\cap E$. Then $C$ belongs to the $\alpha$'th functionally additive (multiplicative) class in $X$ as the intersection of two sets of the same class.
\end{proof}

\begin{proposition}\label{prop:examples}
 Let  $X$ be a topological space, $E\subseteq X$ and
 \begin{enumerate}
    \item $X$ is perfectly normal, or

    \item $X$ is completely regular and $E$ is its Lindel\"{o}f subset, or

    \item $E$ is a functionally open subset of $X$, or

    \item $X$ is a normal space and $E$ is its $F_\sigma$-subset,
 \end{enumerate}
 then $E$ is $0$-embedded in $X$.
\end{proposition}

\begin{proof}
Let $G$ be  a functionally open set in $E$.

{\it (i)}. Choose an open set $U$ in $X$ such that $G=E\cap U$. Then  $U$ is functionally open in $X$ by Vedenissoff's theorem~\cite[p.~45]{Eng}.

{\it (ii)}. Let $U$ be an open set in $X$ such that $G=E\cap U$. Since $X$ is completely regular, $U=\bigcup\limits_{s\in S}U_s$, where $U_s$ is a functionally open set in  $X$ for each $s\in S$. Notice that $G$ is Lindel\"{o}f, provided $G$ is $F_\sigma$ in the Lindel\"{o}f space $E$~\cite[p.~192]{Eng}. Then there exists a countable set $S_0\subseteq S$ such that $G\subseteq \bigcup\limits_{s\in S_0} U_s$. Let $V=\bigcup\limits_{s\in S_0} U_s$. Then $V$ is functionally open in $X$ and $V\cap E=G$.

{\it (iii)}. Consider continuous functions $\varphi:E\to [0,1]$ and $\psi:X\to [0,1]$ such that \mbox{$G=\varphi^{-1}((0,1])$} and $E=\psi^{-1}((0,1])$.  For each
      $x\in X$ we set
      $$
      f(x)=\left\{\begin{array}{ll}
                    \varphi(x)\cdot \psi(x), & x\in E, \\
                    0, & x\in X\setminus E.
                  \end{array}
      \right.
      $$
 Since $\varphi(x)\cdot \psi(x)=0$ on $\overline{E}\setminus E$, $f:X\to [0,1]$ is continuous. Moreover, $G=f^{-1}((0,1])$. Hence, the set $G$ is functionally open in $X$.

{\it (iv)}. Let $\tilde{G}$ be an open set in $X$ such that $G=\tilde G\cap E$. Since $G$ is functionally open in $E$, $G$ is $F_\sigma$ in $E$. Consequently, $G$ is $F_\sigma$ in $X$, provided $E$ is $F_\sigma$ in $X$. Since $X$ is normal, for every $n\in\mathbb N$ there exists a continuous function $f_n:X\to [0,1]$ such that $f_n(x)=1$ if $x\in F_n$ and $f_n(x)=0$ if $x\in X\setminus\tilde G$. Then the set $V=\bigcup\limits_{n=1}^\infty f_n^{-1}((0,1])$ is functionally open in $X$ and $V\cap E=G$.
\end{proof}

Examples~\ref{Nemytzki} and~\ref{Alex} show that none of the conditions (i)--(iv) on $X$ and  $E$ in Proposition~\ref{prop:examples} can be weaken.

Recall that a topological space $X$ is said to be {\it perfect} if every its closed subset is~$G_\delta$ in $X$.

\begin{example}\label{Nemytzki}
  There exist a perfect completely regular space $X$ and its functionally closed subspace $E$ which is not $\alpha$-embedded in $X$ for every $0\le\alpha<\omega_1$.

  Consequently, there is a bounded continuous function on $E$ which cannot be extended to a $\mathcal K_\alpha$-function for every $\alpha$.
\end{example}

\begin{proof}
  Let $X$ be the Niemytski plane~\cite[p.~22]{Eng}, i.e. $X=\mathbb R\times [0,+\infty)$, where a base of neighborhoods of $(x,y)\in X$ with $y>0$ form open balls with the center in $(x,y)$, and a base of neighborhoods of $(x,0)$ form the sets $U\cup\{(x,0)\}$, where $U$ is an open ball which tangent to $\mathbb R\times \{0\}$ in the point $(x,0)$.
  It is well-known that the space $X$ is perfect and completely regular,  but is not normal.

  Denote $E=\mathbb R\times \{0\}$. Since the function $f:X\to \mathbb R$, $f(x,y)=y$, is continuous and  $E=f^{-1}(0)$, the set $E$ is functionally closed in
  $X$.

  Notice that every function $f:E\to\mathbb R$ is continuous. Therefore, \mbox{$|{\mathcal B}_\alpha^*(E)|=2^{2^{\omega_0}}$} for every   $0\le\alpha<\omega_1$. On the other hand,  $|{\mathcal B}_\alpha^*(X)|=2^{\omega_0}$ for every $0\le\alpha<\omega_1$, provided the space $X$ is separable. Hence, for every $0\le\alpha<\omega_1$ there exists a set $A\in {\mathcal B}_\alpha^*(E)$ which can not be extend to a set $B\in {\mathcal B}_\alpha^*(X)$.

  Observe that a function $f:E\to [0,1]$, such that $f=1$ on $A$ and $f=0$ on $E\setminus A$, is continuous on $E$. But there is no $K_\alpha$-function $f:X\to [0,1]$ such that $g|_E=f$, since otherwise the set $B=g^{-1}(1)$ would be an extension of $A$.
\end{proof}

\begin{example}\label{Alex}
  There exist a compact Hausdorff space $X$ and its open subspace $E$ which is not $\alpha$-embedded in $X$ for every $0\le\alpha<\omega_1$.
\end{example}

\begin{proof}
  Let $X=D\cup\{\infty\}$ be the Alexandroff compactification of an uncountable discrete space $D$ \cite[p.~169]{Eng} i $E=D$. Fix $0\le\alpha<\omega_1$ and choose an arbitrary uncountable set $A\subseteq E$ with uncountable complement $X\setminus A$.  Evidently, $A$ is functionally closed in $E$. Assume that there is a set $B$ of the $\alpha$'th functionally multiplicative class in $X$ such that $A=B\cap E$. Clearly, $B=A\cup\{\infty\}$. Moreover, there exists a function $f:X\to \mathbb R$ of the $\alpha$'th Baire class such that  $B=f^{-1}(0)$ \cite[Lemma 2.1]{AGT}. But every continuous function on  $X$, and consequently every Baire function of the class $\alpha$ on $X$ satisfies the equality $f(x)=f(\infty)$ for all but countably many points $x\in X$, which implies a contradiction.
\end{proof}

\begin{proposition}\label{0_imply_alpha}
  Let $0\le\alpha\le\beta<\omega_1$ and let $X$ be a topological space. Then every $\alpha$-embedded subset of $X$ is $\beta$-embedded.
\end{proposition}

\begin{proof} Let $E$ be an $\alpha$-embedded subset of $X$. If $\beta=\alpha$, the assertion of the proposition if obvious.
  Suppose the assertion is true for all  $\alpha\le\beta<\xi$ and let $A$ be a set of the $\xi$'th functionally additive class in $E$. Then there exists a sequence of sets $A_n$ of functionally multiplicative classes $<\xi$ in $E$ such that
$A=\bigcup\limits_{n=1}^\infty A_n$. According to the assumption, for every $n\in\mathbb N$ there is a set $B_n$ of a functionally multiplicative class $<\xi$ in
$X$ such that $A_n=B_n\cap E$. Then the set $B=\bigcup\limits_{n=1}^\infty B_n$ belongs to the $\xi$'th functionally additive class in $X$ and
$A=B\cap E$.
\end{proof}

The inverse proposition is not true, as the following result shows.

\begin{theorem}\label{exist_alpha}
  There exist a completely regular space $X$ and its $1$-embedded subspace $E\subseteq X$ which is not $0$-embedded in $X$.
\end{theorem}

\begin{proof}
Let $X_0=[0,1]$, $X_s=\mathbb N$ for every $s\in (0,1]$, $Y=\prod\limits_{s\in (0,1]}X_s$ and
$$
X=[0,1]\times Y=\prod\limits_{s\in [0,1]}X_s.
$$
Then $X$ is completely regular as a product of completely regular spaces $X_s$.
Let
$$
A_1=(0,1]\quad\mbox{and}\quad A_2=\{0\}.
$$
For $i=1,2$ we consider the set
$$
F_i=\bigcap\limits_{n\ne i}\{y=(y_s)_{s\in (0,1]}\in Y: |\{s\in (0,1]: y_s=n\}|\le 1\}.
$$
Obviously, $F_1\cap F_2=\O$ and the sets $F_1$ and $F_2$ are closed in $Y$.

Let
$$
B_1=A_1\times F_1, \quad B_2=A_2\times F_2\,\,\,\mbox{and} \,\,\,E=B_1\cup B_2.
$$
It is easy to see that the sets $B_1$ and $B_2$ are closed in $E$, and consequently they are functionally clopen in  $E$.

\bigskip
{\it Claim 1.} {\it The set $B_i$ is $0$-embedded in $X$ for every $i=1,2$.}

{\it Proof.} Let $C$ be a functionally open set in $B_1$.

Let us consider the set
$$
H=\{x=(x_s)_{s\in [0,1]}\in X: |\{s\in [0,1]: x_s\ne 1\}|\le\aleph_0\}.
$$
Then the set $[0,1]\times F_i$ is closed in $H$ for every $i=1,2$. Since  $H$ is the $\Sigma$-product of the family $(X_s)_{s\in [0,1]}$ (see
\cite[p.~118]{Eng}), according to \cite{KombMal} the space $H$ is normal. Consequently, $[0,1]\times F_i$ is normal as closed subspace of normal space for every $i=1,2$. Clearly, $B_1$ is functionally open in $[0,1]\times F_1$. Hence, $B_1$ is $0$-embedded in $[0,1]\times F_1$ according to Proposition~\ref{prop:examples}(iii).
 Then $C$ is functionally open in  $[0,1]\times F_1$ by Proposition~\ref{alpha_in_alpha}.
Notice that   the set $[0,1]\times F_1$ is $0$-embedded in $H$ by Propositions~\ref{prop:examples} (iv). Hence, there exists a functionally open set $C'$ in $H$ such that \mbox{$C'\cap ([0,1]\times F_1)=C$}. It follows from \cite{Cors} that $H$ is $0$-embedded in $X$. Then there exists a functionally open set $C''$  in $X$ such that $C''\cap H=C'$. Evidently, $C''\cap B_1=C$. Therefore, the set $B_1$ in  $0$-embedded in $X$.

Analogously, it can be shown that the set $B_2$ is  $0$-embedded  in $X$, using the fact that $B_2$ is $0$-embedded in $[0,1]\times F_2$ according to Proposition~\ref{prop:examples}(iv).

\bigskip
{\it Claim 2.} {\it  The set $E$ is not $0$-embedded in $X$.}

{\it Proof.} Assuming the contrary, we choose a functionally closed set $D$  in $X$ such that
$D\cap E=B_1$. Then $D=f^{-1}(0)$ for some continuous function $f:X\to [0,1]$.
It follows from \cite[p.~117]{Eng} that there exists a countable set $S=\{0\}\cup T$, where $T\subseteq (0,1]$, such that for any $x=(x_s)_{s\in [0,1]}$ and
$y=(y_s)_{s\in [0,1]}$ of $X$ the equality $x|_{S}=y|_S$ implies $f(x)=f(y)$.
Let $y_0\in Y$ be such that $y_0|_{T}$ is a sequence of different natural numbers which are not equal to $1$ or $2$. We choose  $y_1\in F_1$ and $y_2\in F_2$ such that $y_0|_T=y_1|_T=y_2|_T$. Then $$f(a,y_0)=f(a,y_1)=f(a,y_2)$$ for all $a\in [0,1]$.
We notice that $f(0,y_1)=0$. Therefore, $f(0,y_0)=0$. But $f(a,y_2)>0$ for all $a\in A_2$. Then $f(a,y_0)>0$ for all $a\in A_2$. Hence, \mbox{$A_1=(f^{y_0})^{-1}(0)$}, where $f^{y_0}(a)=f(a,y_0)$ for all $a\in [0,1]$, and $f^{y_0}$ is continuous. Thus, the set $A_1=(0,1]$ is closed  in $[0,1]$, which implies a contradiction.

\bigskip
{\it Claim 3.} {\it The set $E$ is $1$-embedded in $X$.}

{\it Proof.} Let $C$ be a functionally $G_\delta$-set in $E$. We put
$$
E_1=A_1\times Y,\quad E_2=A_2\times Y.
$$
Then the set $E_1$ is functionally open in $X$ and the set $E_2$ is functionally closed in~$X$.
For $i=1,2$ let $C_i=C\cap B_i$. Since for every $i=1,2$ the set $C_i$ is functionally $G_\delta$ in the $0$-embedded in $X$ set $B_i$, by Proposition~\ref{0_imply_alpha} there exists a functionally $G_\delta$-set $\tilde C_i$  in $X$ such that $\tilde C_i\cap
B_i=C_i$.  Let
$$
\tilde C=(\tilde C_1\cap E_1)\cup (\tilde C_2\cap E_2).
$$
Then $\tilde C$ is functionally $G_\delta$  in $X$ and $\tilde C\cap E=C$.
\end{proof}

\section{$\alpha$-separated sets and $\alpha$-separated spaces}

Let $0\le\alpha<\omega_1$. Subsets $A$ and $B$ of a topological space $X$ are said to be {\it $\alpha$-separated} if there exists a function $f\in K_\alpha(X)$ such that
$$
A\subseteq  f^{-1}(0)\quad\mbox{and}\quad  B\subseteq f^{-1}(1).
$$
Remark that $0$-separated sets are also called  {\it completely separated} \cite[p.~42]{Eng}.

\begin{lemma}[Lemma 2.1 \cite{K4}]\label{8Lemma21}
  Let $X$ be a topological space, $\alpha>0$ and let $A\subseteq X$ be a subset of the $\alpha$'th functionally additive class. Then there exists a sequence $(A_n)_{n=1}^\infty$ such that each $A_n$ is functionally ambiguous of the class $\alpha$ in $X$, \mbox{$A_n\cap A_m=\O$} for $n\ne m$ and $A=\bigcup\limits_{n=1}^\infty A_n$.
\end{lemma}

\begin{proof}
  Since $A$ belongs to the $\alpha$'th functionally additive class, $A=\bigcup\limits_{n=1}^\infty B_n$, where each $B_n$ belongs to the functionally multiplicative class $<\alpha$ in $X$. Therefore, each $B_n$ is functionally ambiguous of the class $\alpha$. Let $A_1=B_1$ and $A_n=B_n\setminus \bigcup\limits_{k<n} B_k$ for $n>1$. Then $(A_n)_{n=1}^\infty$ is the required sequence.
\end{proof}

\begin{lemma}[Lemma 2.2 \cite{K4}]\label{8Lemma22}
  Let $X$ be a topological space, $\alpha\ge 0$ and let $A_n$ belongs to the $\alpha$'th functionally additive class in $X$ for every $n\in\mathbb N$ with \mbox{$X=\bigcup\limits_{n=1}^\infty A_n$}. Then there exists a sequence $(B_n)_{n=1}^\infty$ of mutually disjoint functionally ambiguous sets of the class $\alpha$ in $X$ such that $B_n\subseteq A_n$ and $X=\bigcup\limits_{n=1}^\infty B_n$.
\end{lemma}

\begin{proof}
  If follows from Lemma~\ref{8Lemma21} that for every $n\in\mathbb N$ there exists a sequence $(F_{n,m})_{m=1}^\infty$ such that each $F_{n,m}$ is functionally ambiguous of the class $\alpha$ in $X$, \mbox{$F_{n,m}\cap F_{n,k}=\O$} for $m\ne k$ and $A_n=\bigcup\limits_{m=1}^\infty F_{n,m}$. Let $k:\mathbb N^2\to\mathbb N$ be a bijection. Set $$C_{n,m}=F_{n,m}\setminus \bigcup\limits_{k(p,s)<k(n,m)}F_{p,s}.$$ Evidently, $\bigcup\limits_{n,m=1}^\infty C_{n,m}=X$. Let $B_n=\bigcup\limits_{m=1}^\infty C_{n,m}$. Then $\bigcup\limits_{n=1}^\infty B_n=\bigcup\limits_{n=1}^\infty A_n=X$ and $B_n\subseteq\bigcup\limits_{m=1}^\infty F_{n,m}=A_n$. Notice that each $C_{n,m}$ is functionally ambiguous of the class $\alpha$. Therefore, $B_n$ belongs to the functionally additive class $\alpha$ for every $n$. Moreover, $B_n\cap B_m=\O$ for $n\ne m$. Since $X\setminus B_n=\bigcup\limits_{k\ne n}B_k$, $B_n$ is functionally ambiguous of the class $\alpha$.
\end{proof}

\begin{lemma}\label{lemma:preimage}
  Let $0\le\alpha<\omega_1$ and let $A$ be a subset of the $\alpha$'th functionally multiplicative class of a topological space $X$. Then there exists a function \mbox{$f\in K_\alpha^*(X)$} such that $A=f^{-1}(0)$.
\end{lemma}

\begin{proof} For $\alpha=0$ the lemma implies from the definition of a functionally closed set.

Let $\alpha>0$. Since the set $B=X\setminus A$ is of the $\alpha$'th functionally additive class, there exists a sequence of functionally ambiguous sets $B_n$ of the $\alpha$'th class in $X$ such that $B=\bigcup\limits_{n=1}^\infty B_n$ and $B_n\cap B_m=\O$ for all $n\ne m$ by Lemma~\ref{8Lemma21}. Define a function $f:X\to [0,1]$,
$$
f(x)=\left\{\begin{array}{ll}
              0, & \mbox{if}\,\, x\in A, \\
              \frac 1n, & \mbox{if}\,\, x\in B_n.
            \end{array}
\right.
$$
Take an arbitrary open set $V\subseteq [0,1]$. If $0\not\in V$ then $f^{-1}(V)$ is of the $\alpha$'th functionally additive class as a union of at most countably many sets $B_n$. If $0\in V$ then there exists such a number $N$ that $\frac 1 n\in V$ for all $n> N$. Then
the set $X\setminus f^{-1}(V)=\bigcup\limits_{n=1}^{N} B_n$ belongs to the $\alpha$'th functionally multiplicative class. Hence, $f^{-1}(V)$ is of the $\alpha$'th functionally additive class in $X$. Therefore, $f\in K_\alpha^*(X)$.
\end{proof}

\begin{proposition}\label{lemma:preimage2}
   Let $0\le \alpha<\omega_1$ and let $X$ be a topological space. Then any two disjoint sets $A$ and $B$ of the $\alpha$'th functionally multiplicative class in $X$ are $\alpha$-separated.
\end{proposition}

\begin{proof}
  By Lemma~\ref{lemma:preimage} we choose functions $f_1,f_2\in K_\alpha(X)$ such that \mbox{$A=f_1^{-1}(0)$} and $B=f_2^{-1}(0)$. For all $x\in X$
  let
  $$
  f(x)=\frac{f_1(x)}{f_1(x)+f_2(x)}.
  $$
  It is easy to see that $f\in K_\alpha(X)$, $f(x)=0$ on $A$ and $f(x)=1$ on $B$.
\end{proof}

Let $0\le\alpha<\omega_1$. A topological space $X$ is {\it $\alpha$-separated} if any two disjoint sets $A,B\subseteq X$ of the $\alpha$'th multiplicative class in  $X$ are $\alpha$-separated. It follows from Urysohn's Lemma \cite[p.~41]{Eng} that a topological space is $0$-separated if and only if it is normal. Proposition~\ref{lemma:preimage2} implies that every perfectly normal space is  $\alpha$-separated for each $\alpha\ge 0$. It is naturally to ask whether there is an $\alpha$-separated space for $\alpha\ge 1$ which is not perfectly normal.

\begin{example}
  There exists a completely regular $1$-separated space which is not perfectly normal.
\end{example}

\begin{proof}
  Let $D=D(\mathfrak m)$ be a discrete space of the cardinality $\mathfrak m$, where $\mathfrak m$ is a measurable cardinal number \cite[12.1]{GillJer}. According to \cite[12.2]{GillJer}, $D$ is not a realcompact space. Let $X=\upsilon D$ be a Hewitt realcompactification of $D$ \cite[p.~218]{Eng}. Then $X$ is an extremally disconnected $P$-space, which is not discrete~\cite[12H]{GillJer}. Thus, there exists a point $x\in X$ such that the set $\{x\}$ is not open.  Then $\{x\}$, being a closed set, is not a $G_\delta$-set, since  $X$ is a $P$-space (i.e. a space in which every $G_\delta$-subset is open). Therefore, the space $X$ is not perfect.

  If $A$ and $B$ are disjoint $G_\delta$-subsets of $X$, then $A$ and $B$ are open in $X$. Notice that in an extremally disconnected space any two disjoint open sets are completely separated \cite[1H]{GillJer}. Consequently, $A$ and $B$ are $1$-separated, since every continuous function belongs to the first Lebesgue class.
\end{proof}

Clearly, every ambiguous set $A$ of the class $0$ in a topological space (i.e., every clopen set) is a functionally ambiguous set of the class $0$. If $A$ is an ambiguous set of the first class, i.e. $A$ is an $F_\sigma$- and a $G_\delta$-set, then $A$ need not be a functionally $F_\sigma$- or a functionally $G_\delta$-set. Indeed, let $X$ be the Niemytski plane,  $E$ be a set which is not of the $G_{\delta\sigma}$-type in $\mathbb R$ and let $A=E\times \{0\}$ be a subspace of $X$. Then $A$ is closed and consequently $G_\delta$-subset of $X$, since the Niemytski plane is a perfect space.  Assume that $A$ is a functionally $F_\sigma$-set in $X$. Then $A=\bigcup\limits_{n=1}^\infty A_n$, where $A_n$ is a functionally closed subset of $X$ for every $n\in\mathbb N$. According to \cite[Theorem 5.1]{Ohta}, a closed subset $F$ of $X$ is a functionally closed set in $X$ if and only if the set $\{x\in\mathbb R: (x,0)\in F\}$ is a $G_\delta$-set in $\mathbb R$. It follows that for every $n\in\mathbb N$ the set $A_n$ is a $G_\delta$-subset of $\mathbb R$, which implies a contradiction.

\begin{theorem}\label{properties}
  \renewcommand{\theenumi}{\arabic{enumi}}
Let $0\le\alpha<\omega_1$ and let $X$ be an $\alpha$-separated space.
\begin{enumerate}
  \item    Every ambiguous set $A\subseteq X$ of the class $\alpha$ is functionally ambiguous of the class $\alpha$.

  \item\label{insertion} For any disjoint sets $A$ and $B$ of the $(\alpha+1)$'th additive class in $X$ there exists a set $C$ of the $(\alpha+1)$'th functionally multiplicative class such that
$$A\subseteq C\subseteq X\setminus B.$$

\item\label{amb} Every  ambiguous set $A$ of the $(\alpha+1)$'th class in  $X$ is a functionally ambiguous set of the $(\alpha+1)$'th class.

\item Any set of the $\alpha$'th multiplicative class in $X$ is  $\alpha$-embedded.
\end{enumerate}
\end{theorem}

\begin{proof}
 (1) Since the set $B=X\setminus A$ belongs to the $\alpha$'th multiplicative class in $X$, there exists a function $f\in K_\alpha(X)$ such that $A\subseteq f^{-1}(0)$ and $B\subseteq f^{-1}(1)$. Then $A=f^{-1}(0)$ and $B=f^{-1}(1)$. Hence, the sets $A$ and $B$ are of the $\alpha$'th functionally multiplicative class. Consequently, $A$ is a functionally ambiguous set of the class~$\alpha$.

 (2)  Choose two sequences $(A_n)_{n=1}^\infty$ and $(B_n)_{n=1}^\infty$, where $A_n$ and $B_n$ belong to the $\alpha$'th multiplicative class in $X$ for every $n\in\mathbb N$, such that $A=\bigcup\limits_{n=1}^\infty A_n$ and $B=\bigcup\limits_{n=1}^\infty B_n$. Since $X$ is $\alpha$-separated, for every $n,m\in \mathbb N$ there exists a function $f_{n,m}\in K_\alpha(X)$ such that $A_n\subseteq f_{n,m}^{-1}(1)$ and $B_m\subseteq f_{n,m}^{-1}(0)$.
 Set $$C=\bigcap\limits_{n=1}^\infty\bigcup\limits_{m=1}^\infty f_{n,m}^{-1}((0,1]).$$
  Then the set $C$ is of the $(\alpha+1)$'th functionally multiplicative class in $X$ and $A\subseteq C\subseteq X\setminus B$.

 (3)   Let $A\subseteq X$ be an  ambiguous set of the $(\alpha+1)$'th class. Denote $B=X\setminus A$. Since $A$ and $B$ are disjoint sets of the $(\alpha+1)$'th additive class in $X$, according to~(\ref{insertion}) there exists a set $C\subseteq X$ of the $(\alpha+1)$'th functionally multiplicative class such that $A\subseteq C\subseteq X\setminus B$. It follows that $A=C$, consequently $A$ is   of the $(\alpha+1)$'th functionally multiplicative class. Analogously, it can be shown that $B$ is also of the $(\alpha+1)$'th functionally multiplicative class. Therefore, $A$ is a functionally ambiguous set of the $(\alpha+1)$'th class.

 (4) If $\alpha=0$ then $X$ is a normal space. Therefore, any closed set $F$ in $X$ is $0$-embedded by Proposition~\ref{prop:examples}.

Let $\alpha>0$ and let $E\subseteq X$ be a set of the  $\alpha$'th multiplicative class in $X$. Choose any set $A$ of the $\alpha$'th functionally multiplicative class in $E$. Since the set $E\setminus A$ belongs to the $\alpha$'th functionally additive class in $E$, there exists a sequence of sets $B_n$ of the $\alpha$'th functionally multiplicative class in $E$ such that $E\setminus A=\bigcup\limits_{n=1}^\infty B_n$.
Then for every $n\in\mathbb N$ the sets $A$ and $B_n$ are disjoint and belong to the $\alpha$'th multiplicative class in $X$. Since $X$ is $\alpha$-separated, we can choose a function $f_n\in K_\alpha(X)$ such that $A\subseteq f_n^{-1}(0)$ and $B_n\subseteq f_n^{-1}(1)$. Let $\tilde
  A=\bigcap\limits_{n=1}^\infty f_n^{-1}(0)$. Then the set $\tilde A$ belongs to the $\alpha$'th functionally multiplicative class in $X$ and $\tilde A\cap E=A$.
\end{proof}

\begin{proposition}\label{norm}
A topological space  $X$ is normal if and only if every its closed subset is $0$-embedded.
\end{proposition}

\begin{proof} We only need to prove the sufficiency. Let $A$ and $B$ be disjoint closed subsets of $X$. Then $A$ is a functionally closed subset of $E=A\cup B$. Since $E$ is closed in $X$, $E$ is a $0$-embedded set. Therefore, there is a functionally closed set $\tilde A$ in $X$ such that $A=E\cap \tilde A$. Then $B$ is a functionally closed subset of the closed set $D=\tilde A\cup B$. Since $D$ is $0$-embedded in $X$, there exists a functionally closed set $\tilde B$ in $X$ such that $B=D\cap\tilde B$. It is easy to check that $\tilde A\cap \tilde B=\O$.  If $f:X\to [0,1]$ be a continuous function such that $\tilde A=f^{-1}(0)$ and $\tilde B=f^{-1}(1)$, then the sets $U=f^{-1}([0,1/2))$ and $V=f^{-1}((1/2,1])$ are disjoint and open in $X$, $A\subseteq U$ and $B\subseteq V$. Hence, $X$ is a normal space.
\end{proof}

An analog of the previous proposition takes place for hereditarily $\alpha$-separated spaces. We say that a~topological space $X$ is {\it hereditarily $\alpha$-separated} if every its subspace is $\alpha$-separated.

\begin{proposition} Let $0\le\alpha<\omega_1$ and let $X$ be a a hereditarily $\alpha$-separated space. If every subset of the $(\alpha+1)$'th multiplicative class in $X$ is $(\alpha+1)$-embedded, then $X$ is $(\alpha+1)$-separated.
\end{proposition}

\begin{proof}
  Let $A, B\subseteq X$ be disjoint sets of the $(\alpha+1)$'th multiplicative class. Then $A$ is ambiguous of the class $(\alpha+1)$ in $E=A\cup B$. Since $E$ belongs to the $(\alpha+1)$'th multiplicative class in $X$, $E$ is $(\alpha+1)$-embedded. Moreover, $E$ is $\alpha$-separated as a subspace of the hereditarily $\alpha$-separated space $X$. According to Theorem~\ref{properties}(3) $A$ is functionally ambiguous of the $(\alpha+1)$'th class in $E$.
  Therefore, there is a  set $\tilde A$ of the $(\alpha+1)$'th functionally multiplicative class in $X$ such that $A=E\cap \tilde A$. Then $B$ is a functionally ambiguous subset of the class $(\alpha+1)$ in $D=\tilde A\cup B$. Since $D$ belongs to the $(\alpha+1)$'th multiplicative class in $X$, $D$ is $(\alpha+1)$-embedded. Therefore, there exists a set $\tilde B$ of the $(\alpha+1)$'th functionally multiplicative class in $X$ such that $B=D\cap\tilde B$. It is easy to check that $\tilde A\cap \tilde B=\O$.   Hence,  the sets $\tilde A$ and $\tilde B$ are $(\alpha+1)$-separated by Proposition~\ref{lemma:preimage2}. Then $A$ and $B$ are $(\alpha+1)$-separated too.
\end{proof}

Remark that the Alexandroff compactification of the real line $\mathbb R$ endowed with the discrete topology is a hereditarily normal space which is not $1$-separated.

We give some examples below of $\alpha$-separated subsets of a completely regular space.

\begin{proposition}\label{prop:lind}
  Let $X$ be a completely regular space and $A,B\subseteq X$ are disjoint sets. Then
  \renewcommand{\theenumi}{\alph{enumi}}
  \begin{enumerate}

    \item if $A$ and $B$ are Lindel\"{o}f $G_\delta$-sets, then they are $1$-separated;

    \item if $A$ is a Lindel\"{o}f hereditarily Baire space and $B$ is a functionally $G_\delta$-set, then $A$ and $B$ are $1$-separated;

    \item if $A$ is Lindel\"{o}f and $B$ is an $F_\sigma$-set, then $A$ and $B$ are $2$-separated.
  \end{enumerate}
\end{proposition}

\begin{proof}
{(a)}. Let $A=\bigcap\limits_{n=1}^\infty U_n$, where $U_n$ is an open set in $X$ for every $n\in\mathbb N$. Since $X$ is completely regular, $U_n=\bigcup\limits_{s\in S_n} U_{s,n}$ for every $n\in\mathbb N$, where all the sets $U_{s,n}$ are functionally open in $X$. Then for every  $n\in\mathbb N$ there is a countable set $S_{n,0}\subseteq S_n$ such that $A\subseteq \bigcup\limits_{s\in S_{n_0}}U_{s,n}$, since $A$ is Lindel\"{o}f. Let  $V_n=\bigcup\limits_{s\in S_{n_0}}U_{s,n}$, $n\in\mathbb N$. Obviously, every $V_n$ is a functionally open set and  $A=\bigcap\limits_{n=1}^\infty V_n$. Hence, $A$ is a functionally $G_\delta$-subset of $X$. Analogously, $B$ is also a functionally  $G_\delta$-set. Therefore,  the sets $A$ and $B$ are $1$-separated by Proposition~\ref{lemma:preimage2}.

{(b)}. According to  \cite[Proposition 12]{Ka} there is a functionally $G_\delta$-set $C$ in $X$ such that $A\subseteq C\subseteq X\setminus B$. Taking a function $f\in K_1(X)$  such that $C=f^{-1}(0)$ and $B=f^{-1}(1)$, we obtain that $A$ and $B$ are $1$-separated.

{(c)}. Let $X\setminus B=\bigcap\limits_{n=1}^\infty U_n$, where $(U_n)_{n=1}^\infty$ is a sequence of open subsets of  $X$. Then $U_n=\bigcup\limits_{s\in S_n} U_{s,n}$ for every $n\in\mathbb N$, where all the sets $U_{s,n}$ are functionally open in $X$. Since $A$ is Lindel\"{o}f, $A\subseteq V_n=\bigcup\limits_{s\in S_{n_0}}U_{s,n}$, where the set  $S_{n_0}$ is countable for every $n\in\mathbb N$. Denote $C=\bigcap\limits_{n=1}^\infty V_n$. Then $C$ is a functionally $G_\delta$-set in $X$ and
$A\subseteq C\subseteq X\setminus B$. Since $C$ is a functionally ambiguous set of the second class, $A$ and $B$ are $2$-separated.
\end{proof}

The following example shows that the class of separation of sets  $A$ and $B$ in Proposition~\ref{prop:lind}(c) can not be made lower.

\begin{example}
  There exist a metrizable space $X$ and its disjoint Lindel\"{o}f $F_\sigma$-subsets $A$ and $B$, which are not $1$-separated.
\end{example}

\begin{proof}
  Let $X=\mathbb R$, $A=\mathbb Q$ and $B$ is a countable dense subsets of irrational numbers. Assume that $A$ and $B$ are $1$-separated, i.e. there exist disjoint $G_\delta$-sets $C$ and $D$ in $\mathbb R$ such that $A\subseteq C$ and $B\subseteq D$. Then $\overline{C}=\overline{D}=\mathbb R$, which implies a contradictions, since  $X$ is a Baire space.
\end{proof}

\section{Ambiguously $\alpha$-embedded sets}

Let $0<\alpha<\omega_1$. A subset $E$ of a topological space $X$ is {\it ambiguously $\alpha$-embedded in $X$} if for any functionally ambiguous set $A$ of the class $\alpha$ in $E$ there exists a functionally ambiguous set  $B$ of the class $\alpha$ in $X$ such that  $A=B\cap E$.

\begin{proposition}\label{lemma:amb3}
 Let $0<\alpha<\omega_1$ and let  $X$ be a topological space. Then every ambiguously $\alpha$-embedded set $E$ in $X$ is $\alpha$-embedded in $X$.
\end{proposition}

\begin{proof}
  Take a set $A\subseteq E$ of the $\alpha$'th functionally additive class in $E$. Then $A$ can be written as $A=\bigcup\limits_{n=1}^\infty A_n$, where $A_n$ is a functionally ambiguous set of the class $\alpha$ in $E$ for every $n\in\mathbb N$ by Lemma~\ref{8Lemma21}. Then there exists a sequence of functionally ambiguous sets $B_n$ of the class $\alpha$ in $X$ such that $A_n=B_n\cap E$ for every $n\in\mathbb N$. Let
  $B=\bigcup\limits_{n=1}^\infty B_n$. Then the set $B$ belongs to the $\alpha$'th functionally additive class in $X$ and $B\cap E=A$.
\end{proof}

We will need the following auxiliary fact.

\begin{lemma}[Lemma 2.3 \cite{K4}]\label{Lemma23} Let $0<\alpha<\omega_1$ and let $X$ be a topological space. Then for any disjoint sets $A,B\subseteq X$ of the $\alpha$'th functionally multiplicative class in $X$ there exists a functionally ambiguous set $C$ of the class $\alpha$ in $X$ such that $A\subseteq C\subseteq X\setminus B$.
\end{lemma}

\begin{proof}
  Lemma~\ref{8Lemma22} implies that there are disjoint functionally ambiguous sets $E_1$ and $E_2$ of the class $\alpha$ such that $E_1\subseteq X\setminus A$, $E_2\subseteq X\setminus B$ and $X=E_1\cup E_2$. It remains to put $C=E_2$.
\end{proof}

\begin{proposition}\label{lemma:amb1}
  Let $0<\alpha<\omega_1$ and let  $X$ be a topological space. Then every $\alpha$-embedded set $E$ of the $\alpha$'th functionally multiplicative class in $X$ is ambiguously $\alpha$-embedded in  $X$.
\end{proposition}

\begin{proof}
  Consider a functionally ambiguous set $A$ of the class $\alpha$ in $E$. Then there exists a set $B$ of the $\alpha$'th functionally multiplicative class in  $X$ such that $A=B\cap E$. Since $E$ is of the $\alpha$'th functionally multiplicative class in $X$, the set $A$ is also of the same class in  $X$. Analogously, the set $E\setminus A$ belongs to the $\alpha$'th functionally multiplicative class in  $X$. It follows from Lemma~\ref{Lemma23} that there exists a functionally ambiguous set $C$ of the class $\alpha$ in $X$ such that $A\subseteq C$ and $C\cap (E\setminus A)=\O$. Clearly, $C\cap E=A$. Hence, the set $E$ is ambiguously $\alpha$-embedded in $X$.
\end{proof}

\begin{example}
  There exists a $0$-embedded $F_\sigma$-set $E\subseteq \mathbb R$ which is not ambiguously $1$-embedded.
\end{example}

\begin{proof}
  Let $E=\mathbb Q$. Obviously, $E$ is a $0$-embedded set. Consider any two disjoint $A$ and $B$ which are dense in $E$. Then $A$ and $B$ are simultaneously $F_\sigma$- and  $G_\delta$-sets in $E$. Assume that there exists an $F_\sigma$- and $G_\delta$-set $C$ in $\mathbb R$ such that $A=E\cap C$. Since $A\subseteq C$ and $B\subseteq \mathbb R\setminus C$, the sets $C$ and $\mathbb R\setminus C$ are dense in $\mathbb R$. Moreover, the sets  $C$ and $\mathbb R\setminus C$ are $G_\delta$ in $\mathbb R$. It implies a contradiction, since $\mathbb R$ is a Baire space.
\end{proof}

\begin{example}
  There exits a Borel non-measurable ambiguously $1$-embedded subset of a perfectly normal compact space.
\end{example}

\begin{proof}
 Let $X$ be the ''two arrows'' space (see \cite[p.~212]{Eng}), i.e. $X=X_0\cup X_1$, where $X_0=\{(x,0):x\in (0,1]\}$ and $X_1=\{(x,1):x\in [0,1)\}$. The topology base on $X$ is generated by the sets
 $$
 ((x-\frac 1n,x]\times \{0\})\cup((x-\frac 1n,x)\times \{1\})\,\,\, \mbox{if} \,\,\, x\in (0,1]\,\,\, \mbox{and} \,\,\,n\in\mathbb N
 $$ and
 $$
 ((x,x+\frac 1n)\times \{0\})\cup([x,x+\frac 1n)\times \{1\}) \,\,\, \mbox{if} \,\,\, x\in [0,1)\,\,\, \mbox{and} \,\,\,n\in\mathbb N.
 $$
 For a set $A\subseteq X$ we denote
 $$
 A^+=\{x\in [0,1]: (x,1)\in A\}\,\,\,\mbox{and}\,\,\, A^-=\{x\in [0,1]:(x,0)\in A\}.
 $$

 It is not hard to verify that for every open or closed set $A\subseteq X$ we have $|A^+\Delta A^-|\le\aleph_0$. It follows that
 $|B^+\Delta B^-|\le\aleph_0$ for any Borel measurable set $B\subseteq X$.

 Let $E=X_0$. Since $E^+=\O$ and $E^-=(0,1]$, the set $E$ is non-measurable. We show that $E$ is an ambiguously $1$-embedded set. Indeed, let $A\subseteq E$ be an $F_\sigma$- and $G_\delta$-subset of $E$. Then  $B=E\setminus A$ is also an $F_\sigma$- and $G_\delta$-subset of $E$. Let $\tilde A$ and $\tilde B$ be $G_\delta$-sets in $X$ such that $A=\tilde A\cap E$ and $B=\tilde B\cap E$. The inequalities $|\tilde A^+\Delta \tilde A^-|\le \aleph_0$ and $|\tilde B^+\Delta \tilde B^-|\le \aleph_0$ imply that $|C|\le\aleph_0$, where $C=\tilde A\cap \tilde B$. Hence, $C$ is an $F_\sigma$-set in $X$.
 Moreover, $C$ is a  $G_\delta$-set in $X$. Therefore, $\tilde A\setminus C$ and $\tilde B\setminus C$ are $G_\delta$-sets in
 $X$. According to Lemma~\ref{Lemma23}, there is an  $F_\sigma$- and $G_\delta$-set $D$ in $X$ such that $\tilde A\setminus C\subseteq D$ and $D\cap (\tilde  B\setminus C)=\O$. Then $D\cap E=A$.
\end{proof}

\section{Extension of real-valued $K_\alpha$-functions}

Analogs of Proposition~\ref{prop:Extension_Of_Baire_Alpha} and Theorem~\ref{Extension_Not_Bounded} for $\alpha=1$ were proved in \cite{Ka}.

\begin{proposition}\label{prop:Extension_Of_Baire_Alpha}
  Let $X$ be a topological space, $E\subseteq X$ and $0< \alpha<\omega_1$. Then the following conditions are equivalent:
  \begin{enumerate}
    \item $E$ is $K_\alpha^*$-embedded in $X$;

    \item $E$ is ambiguously $\alpha$-embedded in $X$;

    \item $(X,E,[c,d])$ has the $K_\alpha$-extension property for any segment $[c,d]\subseteq \mathbb R$.
  \end{enumerate}
\end{proposition}

\begin{proof}

 $(i) \Longrightarrow (ii)$. Take an arbitrary functionally ambiguous set $A$ of the class $\alpha$ in $E$ and consider its characteristic function
 $\chi_A$. Then $\chi_A\in K_\alpha^*(E)$, as is easy to check. Let $f\in K_\alpha(X)$ be en extension of $\chi_A$. Then the sets $f^{-1}(1)$ and
 $f^{-1}(0)$ are disjoint and belong to the $\alpha$'th functionally multiplicative class in $X$. According to Lemma~\ref{Lemma23} there exists a functionally ambiguous set $B$ of the class $\alpha$ in $X$ such that $f^{-1}(1)\subseteq B$ and $B\cap f^{-1}(0)=\O$. It remains to notice that $B\cap  E=f^{-1}(1)\cap E=\chi_A^{-1}(1)=A$. Hence, $E$ is an ambiguously $\alpha$-embedded set in $X$.

 $(ii) \Longrightarrow (iii)$. Let $f\in K_\alpha(E,[c,d])$. Define
$$
h_1(x)=\left\{\begin{array}{ll}
               f(x), & \mbox{if}\,\, x\in E,\\
               {\rm inf} f(E), & \mbox{if}\,\, x\in X\setminus
               E,
             \end{array}
\right.
$$
$$
h_2(x)=\left\{\begin{array}{ll}
               f(x), & \mbox{if}\,\, x\in E,\\
               {\rm sup} f(E), & \mbox{if}\,\, x\in X\setminus
               E,
             \end{array}
\right.
$$
Then $c\le h_1(x)\le h_2(x)\le d$ for all $x\in X$.

We prove that for any reals $a<b$ there exists a function $h\in K_\alpha(X)$ such that
$$
h_2^{-1}([c,a])\subseteq h^{-1}(0)\quad\mbox{and}\quad h_1^{-1}([b,d])\subseteq h^{-1}(1).
$$

Fix $a<b$. Without loss of generality we may assume that
$$
{\rm inf} f(E)\le a<b\le {\rm sup}f(E).
$$
Denote
$$
A_1=f^{-1}([c,a]),\quad
A_2=f^{-1}([b,d]).
$$
Then $A_1$ and $A_2$ are disjoint sets of the $\alpha$'th functionally multiplicative class in~$E$. Using Lemma~\ref{Lemma23}, we choose a functionally ambiguous set $C$ of the class $\alpha$ in $E$ such that $A_1\subseteq C$ and $C\cap A_2=\O$. Since $E$ is an ambiguously $\alpha$-embedded set in $X$, there exists such a functionally ambiguous set $D$ of the class $\alpha$ in $X$ that $D\cap E=C$. Moreover, by Proposition~\ref{lemma:amb3} there exist sets $B_1$ and $B_2$ of the $\alpha$'th functionally multiplicative class in $X$ such that $A_i=E\cap B_i$ when $i=1,2$.  Let
$$
\tilde A_1=D\cap B_1,\quad \tilde A_2=(X\setminus D)\cap B_2.
$$
Then the sets $\tilde{A_1}$ and $\tilde{A_2}$ are disjoint and belong to the $\alpha$'th functionally multiplicative class in $X$. Moreover, $A_1=E\cap\tilde A_1$ and $A_2=E\cap \tilde A_2$. According to Proposition~\ref{lemma:preimage2} there is a function $h\in K_\alpha^*(X)$ such that
$$
h^{-1}(0)=\tilde A_1\quad\mbox{and}\quad h^{-1}(1)=\tilde A_2.
$$
According to \cite[Theorem 3.2]{LMZ} there exists a function $g\in K_\alpha(X)$ such that
$$h_1(x)\le g(x)\le h_2(x)$$ for all $x\in X$. Clearly, $g$ is an extension of $f$ and $g\in K_\alpha(X,[c,d])$.

$(iii)\Longrightarrow (i)$. Let $f\in K_\alpha^*(E)$ and let $|f(x)|\le C$ for all $x\in E$. Consider a function $g\in K_\alpha(X)$ which is an extension of $f$. Define a function $r:\mathbb R\to [-C,C]$, $r(x)=\min\{C,\max\{x,-C\}\}$. Obviously, $r$ is continuous. Let $h=r\circ g$. Then $h\in K_\alpha^*(X)$ and $h|_E=f$. Hence, $E$ is $K_\alpha^*$-embedded in $X$.
\end{proof}

\begin{lemma}\label{lemma:ambbb1}
    Let  $0< \alpha<\omega_1$, $X$ be a topological space and let $E\subseteq X$ be such an $\alpha$-embedded set in $X$ that for any set $A$ of the $\alpha$'th functionally multiplicative class in $X$ such that $E\cap A=\O$ the sets $E$ and $A$ are $\alpha$-separated. Then $E$ is an ambiguously $\alpha$-embedded set.
\end{lemma}

\begin{proof}
   Consider a functionally ambiguous set $C$ of the class $\alpha$ in $E$ and denote $C_1=C$, $C_2=E\setminus C$. Then there exist sets  $\tilde C_1$ and $\tilde C_2$ of the $\alpha$'th functionally multiplicative class in $X$ such that $\tilde C_i\cap E=C_i$ when $i=1,2$. Then the set $A=\tilde C_1\cap\tilde C_2$ is of the $\alpha$'th functionally multiplicative class in $X$ and $A\cap E=\O$. Let $h\in K_\alpha(X)$ be a function such that $E\subseteq h^{-1}(0)$ and $A\subseteq h^{-1}(1)$. Denote $H=h^{-1}(0)$ and $H_i=H\cap \tilde C_i$ when $i=1,2$. Since $H_1$ and $H_2$ are disjoint sets of the $\alpha$'th functionally multiplicative class in $X$, by Lemma~\ref{Lemma23} there is a functionally ambiguous set $D$ of the class $\alpha$ in $X$ such that $H_1\subseteq D\subseteq X\setminus H_2$. Obviously, $D\cap E=C$.
\end{proof}

\begin{theorem}\label{Extension_Not_Bounded}
 Let  $0< \alpha<\omega_1$ and let $E$ be a subset of a topological space $X$. Then the following conditions are equivalent:
  \begin{enumerate}
     \item $E$ is $K_\alpha$-embedded in $X$;

     \item $E$ is $\alpha$-embedded in $X$ and for any set $A$ of the $\alpha$'th functionally multiplicative class in $X$ such that $E\cap A=\O$ the sets $E$ and $A$ are $\alpha$-separated.
  \end{enumerate}
\end{theorem}

\begin{proof}
 $(i) \Longrightarrow (ii)$. Let $C\subseteq E$ be a set of the $\alpha$'th functionally multiplicative class in $E$. Then by Lemma~\ref{lemma:preimage} we choose a function $f\in K_\alpha^*(E)$ such that $C=f^{-1}(0)$. If $g\in K_\alpha(X)$ is an extension of $f$, then the set $B=g^{-1}(0)$ belongs to the $\alpha$'th functionally multiplicative class in $X$ and $B\cap E=C$. Hence, $E$ is an $\alpha$-embedded set in $X$.

 Now consider a set $A$ of the $\alpha$'th functionally multiplicative class in $X$ such that $E\cap A=\O$. According to Lemma~\ref{lemma:preimage} there is a function $h\in K_\alpha^*(X)$ such that $A=h^{-1}(0)$. For all $x\in E$ let $\displaystyle f(x)=\frac{1}{h(x)}$. Then $f\in K_\alpha(E)$. Let $g\in K_\alpha(X)$ be an extension of $f$. For all $x\in X$ let $\varphi(x)=g(x)\cdot h(x)$. Clearly, $\varphi\in K_\alpha(X)$. It is not hard to verify that $E\subseteq \varphi^{-1}(1)$ and $A\subseteq \varphi^{-1}(0)$.

 $(ii)\Longrightarrow (i)$.  Remark that according to Lemma~\ref{lemma:ambbb1} the set $E$ is ambiguously $\alpha$-embedded in $X$.

 Let $f\in K_\alpha(E)$ and let $\varphi:\mathbb R\to (-1,1)$ be a homeomorphism.
 Using Proposition~\ref{prop:Extension_Of_Baire_Alpha} to the function $\varphi\circ f:E\to [-1,1]$ we have that there exists a function $h\in K_\alpha(X,[-1,1])$ such that $h|_E=\varphi\circ f$. Let
  $$
  A=h^{-1}(-1)\cup h^{-1}(1).
  $$
  Then $A$ belongs to the $\alpha$'th functionally multiplicative class in $X$ and $A\cap E=\O$. Therefore, there exists a function $\psi\in K_\alpha(X)$ such that $A\subseteq\psi^{-1}(0)$ and $E\subseteq\psi^{-1}(1)$.
  For all $x\in X$ define
  $$
  g(x)=\varphi^{-1}(h(x)\cdot \psi(x)).
  $$
  Remark that $g\in K_\alpha(X)$ and $g|_E=f$.
  \end{proof}

  \begin{corollary}
     Let $0< \alpha<\omega_1$ and let $E$ be a subset of the $\alpha$'th functionally multiplicative class of a topological space $X$. Then the following conditions are equivalent:
  \begin{enumerate}
     \item $E$ is $K_\alpha$-embedded in $X$;

     \item $E$ is $\alpha$-embedded in $X$.
  \end{enumerate}
  \end{corollary}

\section{$K_1^*$-embedding versus $K_1$-embedding}

A family $\mathcal U$ of non-empty open sets of a space $X$ is called {\it a $\pi$-base} \cite{EGT} if for any non-empty open set $V$ of
$X$ there is $U\in\mathcal U$ with $V\subseteq U$.

\begin{proposition}\label{3}
  Let $X$ be a perfect space of the first category with a countable $\pi$-base.  Then there exist disjoint $F_\sigma$- and $G_\delta$-subsets $A$ and $B$ of $X$ which are dense in $X$ and $X=A\cup B$.
\end{proposition}

\begin{proof}
Let $(V_n:n\in{\mathbb N})$ be a $\pi$-base in $X$ and $X=\bigcup\limits_{n=1}^\infty X_n$, where $X_n$ is a closed nowhere dense subset of~$X$ for every $n\ge 1$. Let $E_1=X_1$ and $E_n=X_n\setminus\bigcup\limits_{k<n}X_k$ for $n\ge 2$. Then $E_n$ is a nowhere dense $F_\sigma$- and $G_\delta$-subset of $X$ for every $n\ge 1$, $E_n\cap E_m=\O$ if $n\ne m$, and  $X=\bigcup\limits_{n=1}^\infty E_n$.

Let $m_0=0$. We choose a number $n_1\ge 1$ such that $(\bigcup\limits_{n=1}^{n_1}E_n)\cap V_1\ne\O$ and let $A_1=\bigcup\limits_{n=1}^{n_1}E_n$. Since $\overline{ X\setminus A_1}=X$, there exists a number $m_1>n_1$ such that $(\bigcup\limits_{n=n_1+1}^{m_1}E_n )\cap V_1\ne\O$. Set $B_1=\bigcup\limits_{n=n_1+1}^{m_1}E_n$. It follows from the equality $\overline{X\setminus (A_1\cup B_1)}=X$ that there exists $n_2>m_1$ such that $(\bigcup\limits_{n=m_1+1}^{n_2}E_n )\cap V_2\ne\O$. Further, there is such $m_2>n_2$ that $(\bigcup\limits_{n=n_2+1}^{m_2}E_n )\cap V_2\ne\O$. Let $A_2=\bigcup\limits_{n=m_1+1}^{n_2}E_n$ and $B_2=\bigcup\limits_{n=n_2+1}^{m_2}E_n$. Repeating this process, we obtain the sequence of numbers
  $$
  m_0<n_1<m_1<\dots<n_k<m_k<n_{k+1}<\dots
  $$
and the sequence of sets
  $$
  A_k=\bigcup\limits_{n=m_{k-1}+1}^{n_k}E_n,\quad B_k=\bigcup\limits_{n=n_k+1}^{m_k}E_n,\quad k\ge 1,
  $$
such that $A_k\cap V_k\ne\O$ and $B_k\cap V_k\ne\O$ for every $k\ge 1$.

Let $A=\bigcup\limits_{k=1}^\infty A_k$ and $B=\bigcup\limits_{k=1}^\infty B_k$. Clearly, $X=A\cup B$, $A\cap B=\O$ and $\overline{A}=\overline{B}=X$. Moreover, $A$ and $B$ are $F_\sigma$-sets in $X$. Therefore, $A$ and $B$ are $F_\sigma$- and $G_\delta$-subsets of $X$.
\end{proof}

We say that a topological space $X$ {\it hereditarily has a countable $\pi$-base} if every its closed subspace has a countable $\pi$-base.

\begin{proposition}\label{prop:hereditarilyBaire}
   Let $X$ be a hereditarily Baire space, $E$ be a perfectly normal ambiguously $1$-embedded subspace of $X$ which hereditarily has a countable $\pi$-base. Then $E$ is a hereditarily Baire space.
\end{proposition}

\begin{proof}
  Assume that $E$ is not a hereditarily Baire space. Then there exists a nonempty closed set $C\subseteq X$ of the first category. Notice that $C$ is a perfectly normal space with a countable $\pi$-base. According to Proposition~\ref{3} there exist disjoint dense $F_\sigma$- and $G_\delta$-subsets $A$ and $B$ of $C$  such that  $C=A\cup B$. Since $C$ is $F_\sigma$- and $G_\delta$-set in $E$, the sets $A$ and $B$ are also $F_\sigma$ and $G_\delta$ in $E$. Therefore there exist disjoint functionally $F_\sigma$- and $G_\delta$-subsets $\tilde A$ and $\tilde B$ of $X$ such that  $A=\tilde A\cap E$ and $B=\tilde B\cap E$. Notice that the sets $\tilde A$ and $\tilde B$ are dense in $\overline{C}$. Taking into account that $X$ is hereditarily Baire, we have that $\overline{C}$ is a Baire space. It follows a contradiction, since $\tilde A$ and $\tilde B$ are disjoint dense $G_\delta$-subsets of $\overline{C}$.
\end{proof}

Remark that there exist a metrizable separable Baire space $X$ and its ambiguously $1$-embedded subspace $E$ which is not a Baire space. Indeed, let  $X=({\mathbb Q}\times \{0\})\cup ({\mathbb R}\times (0,1])$ and  $E={\mathbb Q}\times \{0\}$. Then $E$ is closed in $X$. Therefore, any $F_\sigma$- and $G_\delta$-subset $C$ of $E$ is also $F_\sigma$- and $G_\delta$- in $X$. Hence, $E$ is an ambiguously $1$-embedded set in $X$.

\begin{theorem}
  Let $X$ be a hereditarily Baire space and let $E\subseteq X$ be its perfect Lindel\"{o}f subspace which hereditarily has a countable $\pi$-base. Then $E$ is  $K_1^*$-embedded in $X$ if and only if $E$ is $K_1$-embedded in $X$.
\end{theorem}

\begin{proof} Since the sufficiency is obvious, we only need to prove the necessity.

 According to Proposition~\ref{prop:Extension_Of_Baire_Alpha} the set $E$ is ambiguously $1$-embedded in $X$. Using Proposition~\ref{prop:hereditarilyBaire}, we have $E$ is a hereditarily Baire space. Since $E$ is Lindel\"{o}f, Proposition~\ref{prop:lind} (b) implies that $E$ is $1$-separated from any functionally $G_\delta$-set $A$ of $X$ such that $A\cap E=\O$. Therefore, by Theorem~\ref{Extension_Not_Bounded} the set $E$ is $K_1$-embedded in $X$.
\end{proof}

\section{A generalization of the Kuratowski theorem}

K.~Kuratowski \cite[p.~445]{Ku1} proved that every mapping $f\in K_\alpha(E,Y)$ has an extension $g\in K_\alpha(X,Y)$ if the case $X$ is a metric space, $Y$ is a Polish space and $E\subseteq X$ is a set of the multiplicative class $\alpha>0$.

In this section we will prove that the Kuratowski Extension Theorem is still valid if $X$ is a topological space and $E$ is a $K_\alpha$-embedded subset of $X$.

We say that a subset $A$ of a space $X$ is {\it discrete} if  any point $a\in A$ has a neighborhood $U\subseteq X$ such that $U\cap A=\{a\}$.

\begin{theorem}[Theorem 2.11 \cite{K4}]\label{8Theorem211}
 Let $X$ be a topological space, $Y$ be a metrizable separable space, $0\le\alpha<\omega_1$ and  $f\in K_\alpha(X,Y)$. Then there exists a sequence $(f_n)_{n=1}^\infty$  such that
 \begin{enumerate}
   \item  $f_n\in K_\alpha(X,Y)$ for every $n$;

   \item $(f_n)_{n=1}^\infty$ is uniformly convergent to $f$;

   \item $f_n(X)$ is at most countable and discrete for every $n$.
 \end{enumerate}
\end{theorem}

\begin{proof}
  Consider a metric $d$ on $Y$ which generates its topological structure. Since $(Y,d)$ is metric separable space, for every $n$ there is a subset $Y_n=\{y_{i,n}:i\in I_n\}$ of $Y$ such that $Y_n$ is discrete, $|I_n|\le\aleph_0$ and for any $y\in Y$ there exists $i\in I_n$ such that
$d(y,y_{i,n})<1/n$ (see \cite[p.~226]{Ku1}).

For every $n\in\mathbb N$ and $i\in I_n$ put
$A_{i,n}=\{x\in X: d(f(x),y_{i,n})<1/n\}$. Then each $A_{i,n}$ belongs to the $\alpha$'th functionally additive class in $X$ and
$\bigcup\limits_{i\in I_n}A_{i,n}=X$ for every $n$. According to Lemma~\ref{8Lemma22} for every $n$ we can choose a sequence $(F_{i,n})_{i\in I_n}$ of disjoint functionally ambiguous sets of the class $\alpha$ such that $F_{i,n}\subseteq
A_{i,n}$ and $\bigcup\limits_{i\in I_n}F_{i,n}=X$.

For all $x\in X$ and $n\in\mathbb N$ let $f_n(x)= y_{i,n}$ if $x\in F_{i,n}$ for some $i\in I_n$. Notice that  $f_n\in K_\alpha(X,Y)$ for every $n\in\mathbb N$.

It remains to prove that the sequence $(f_n)_{n=1}^\infty$ is uniformly convergent to  $f$. Indeed, fix  $x\in X$ and $n\in{\mathbb
N}$. Then there exists $i\in I_n$ such that $x\in F_{i,n}$. Since $F_{i,n}\subseteq A_{i,n}$, $d(f(x),f_n(x))=d(f(x),y_{i,n})<\frac{1}{n}$, which completes the proof.
\end{proof}

Recall that a family $(A_s:s\in S)$ of subsets of a topological space $X$ is called {\it a partition of $X$} if $X=\bigcup\limits_{s\in S}A_s$ and $A_s\cap A_t=\O$ for all $s\ne t$.

\begin{proposition}\label{lemma:amb2}
  Let $0<\alpha<\omega_1$, $X$ be a topological space, $E\subseteq X$ be an $\alpha$-embedded set which is $\alpha$-separated from any disjoint with it set of the $\alpha$'th functionally multiplicative class in $X$ and let $(A_n:n\in\mathbb N)$ be a partition of $E$ by functionally ambiguous sets of the class $\alpha$ in $E$. Then there is a partition $(B_n:n\in\mathbb N)$ of $X$ by functionally ambiguous sets of the class $\alpha$ in $X$ such that  $A_n=E\cap B_n$ for every $n\in\mathbb N$.
\end{proposition}

\begin{proof}
 According to Proposition~\ref{lemma:ambbb1} for every $n\in\mathbb N$ there exists a functionally ambiguous set $D_n$ of the class $\alpha$ in $X$ such that $A_n=D_n\cap E$.
By the assumption there exists a function $f\in K_\alpha(X)$ such that $E\subseteq f^{-1}(0)$ and $X\setminus \bigcup\limits_{n=1}^\infty D_n\subseteq f^{-1}(1)$. Let $D=f^{-1}(0)$. Then the set $X\setminus D$ is of the $\alpha$'th functionally additive class in  $X$. Then there exists a sequence $(E_n)_{n=1}^\infty$ of functionally ambiguous set of the class $\alpha$ in $X$ such that $X\setminus
D=\bigcup\limits_{n=1}^\infty E_n$.  For every $n\in\mathbb N$ denote $C_n=E_n\cup D_n$. Then all the sets $C_n$ are functionally ambiguous of the class $\alpha$ in $X$ and $\bigcup\limits_{n=1}^\infty C_n=X$. Let $B_1=C_1$ and
$B_n=C_n\setminus \left(\bigcup\limits_{k<n} C_k\right)$ for
$n\ge 2$. Clearly, every $B_n$ is a functionally ambiguous set of the class $\alpha$ in $X$, $B_n\cap B_m=\O$ if $n\ne m$ and
$\bigcup\limits_{n=1}^\infty B_n=\bigcup\limits_{n=1}^\infty C_n=X$. Moreover, $B_n\cap E=A_n$ for every $n\in\mathbb N$.
\end{proof}

Let $0\le\alpha<\omega_1$, $X$ and  $Y$  be topological spaces and $E\subseteq X$. We say that a collection $(X,E,Y)$ has {\it the $K_\alpha$-extension property} if every mapping $f\in K_\alpha(E,Y)$ can be extended to a mapping $g\in K_\alpha(X,Y)$.

\begin{theorem}\label{th:Extension_H_alpha}
   \renewcommand{\theenumi}{\roman{enumi}}
  Let $0<\alpha<\omega_1$ and let $E$ be a subset of a topological space $X$. Then the following conditions are equivalent:
  \begin{enumerate}
    \item $E$ is $K_\alpha$-embedded in $X$;

    \item $(X,E,Y)$ has the $K_\alpha$-extension property for any Polish space $Y$.
  \end{enumerate}
\end{theorem}

\begin{proof}

Since the implication (ii) $\Rightarrow$ (i) is obvious, we only need to prove the implication (i) $\Rightarrow$ (ii). Let $Y$ be a Polish space with a metric $d$ which generates its topological structure such that $(Y,d)$ is complete and let $f\in K_\alpha(E,Y)$.

It follows from Theorem~\ref{8Theorem211} that there exists a sequence of mappings $f_n\in K_\alpha(E,Y)$ which is uniformly convergent to $f$ on $E$. Moreover, for every $n\in\mathbb N$ the set $f_n(E)=\{y_{i_n,n}: i_n\in I_n\}$ is at most countable and discrete. We may assume that each $f_n(E)$ consists of distinct points.

For every $n\in\mathbb N$ and for each $(i_1,\dots, i_n)\in I_1\times\dots\times I_n$ let
  $$
  B_{i_1,\dots, i_n}=f_1^{-1}(y_{i_1,1})\cap\dots\cap f_n^{-1}(y_{i_n,n}).
  $$
Then for each $i_1\in I_1$, \dots, $i_n\in I_n$ the set $B_{i_1\dots i_n}$ is functionally ambiguous of the class $\alpha$ in $E$ and the family $(B_{i_1,\dots, i_n}:i_1\in I_1,\dots, i_n\in I_n)$ is a partition of $E$ for every $n\in\mathbb N$. By Proposition~\ref{lemma:amb2} we choose a sequence of systems of functionally ambiguous sets $D_{i_1\dots i_n}$ of the class $\alpha$ in $X$ such that
 \renewcommand{\theenumi}{\arabic{enumi}}
  \begin{enumerate}
    \item $D_{i_1,\dots, i_n}\cap E=B_{i_1,\dots, i_n}$ for every $n\in\mathbb N$ and $(i_1,\dots, i_n)\in I_1\times\dots\times I_n$;

    \item $(D_{i_1,\dots, i_n}: i_1\in I_1,\dots,i_n\in I_n)$ is a partition of $X$ for every $n\in\mathbb N$.
  \end{enumerate}
For all $n\in\mathbb N$ and $(i_1, \dots, i_n)\in I_1\times\dots\times I_n$ let
  \begin{enumerate}
  \setcounter{enumi}{2}
   \item $D_{i_1,\dots, i_n}=\O,\,\,\mbox{if}\,\, B_{i_1,\dots, i_n}=\O$.
  \end{enumerate}

Notice that the system $(B_{i_1,\dots, i_n, i_{n+1}}: i_{n+1}\in I_{n+1})$ forms a partition of the set
  $B_{i_1,\dots, i_n}$  for every $n\in\mathbb N$.

For all $i_1\in I_1$ let
  $$
  C_{i_1}=D_{i_1}.
  $$
Assume that for some $n\ge 1$ the system $(C_{i_1,\dots, i_n}:i_1\in I_1,\dots, i_n\in I_n)$ of functionally ambiguous sets of the class  $\alpha$ in $X$ is already defined and
   \renewcommand{\theenumi}{\Alph{enumi}}
   \begin{enumerate}
     \item $B_{i_1,\dots, i_n}=E\cap C_{i_1,\dots, i_n}$;

     \item $(C_{i_1,\dots, i_n}:i_1\in I_1,\dots, i_n\in I_n)$ is a partition of $X$;

     \item $C_{i_1,\dots, i_n}=\O$ if $B_{i_1,\dots, i_n}=\O$;

     \item $(C_{i_1\dots i_{n-1}i_n}:i_n\in I_n)$ is a partition of the set $C_{i_1,\dots, i_{n-1}}$.
   \end{enumerate}
   Fix $i_1,\dots,i_n$. Since the set $K=C_{i_1,\dots,i_n}\setminus \bigcup\limits_{k\in I_{n+1}}D_{i_1,\dots, i_n, k}$ is of the $\alpha$'th functionally multiplicative class in $X$ and $K\cap E=\O$, there exists a set $H$ of the $\alpha$'th functionally multiplicative class in $X$ such that $E\subseteq H\subseteq X\setminus K$. Using \cite[Lemma 2.1]{K4} we obtain that there exists a sequence $(A_k)_{k=1}^\infty$ of disjoint functionally ambiguous sets of the class $\alpha$ in $X$ such that
   $$
   C_{i_1,\dots, i_n}\setminus H=\bigcup\limits_{k=1}^\infty A_k.
   $$
Let
   $$
   M_{i_1,\dots, i_n,i_{n+1}}=\O, \quad\mbox{if}\quad D_{i_1,\dots i_n,i_{n+1}}=\O,
   $$
and
   $$
   M_{i_1,\dots i_n,i_{n+1}}=(A_{i_{n+1}}\cup D_{i_1,\dots, i_n,i_{n+1}})\cap C_{i_1,\dots, i_n}, \quad\mbox{if}\quad D_{i_1,\dots, i_n,i_{n+1}}\ne\O.
   $$
Now let
   $$
   C_{i_1,\dots, i_n, 1}= M_{i_1,\dots, i_n, 1},$$
and
$$
C_{i_1,\dots, i_n, i_{n+1}}=M_{i_1,\dots, i_n, i_{n+1}}\setminus
   \bigcup\limits_{k<i_{n+1}}M_{i_1,\dots, i_n, k} \quad\mbox{if}\,\, i_{n+1}>1.
   $$
Then for every $n\in\mathbb N$ the system $(C_{i_1,\dots, i_n}:i_1\in I_1,\dots,i_n\in I_n)$ of functionally ambiguous sets of the class $\alpha$ in $X$ has the properties (A)--(D).

For each $n\in\mathbb N$ and $x\in X$ let
  $$
  g_n(x)=y_{i_n,n},
  $$
if $x\in C_{i_1,\dots, i_n}$.
It is not hard to prove that $g_n\in K_\alpha(X,Y)$.

We show that the sequence $(g_n)_{n=1}^\infty$ is uniformly convergent on $X$. Indeed, let $x_0\in X$ and $n,m\in\mathbb N$.
Without loss of generality, we may assume that $n\ge m$. By the property~(B), $x_0\in C_{i_1,\dots, i_n}\cap C_{j_1,\dots, j_m}$. It follows from  (B) and (D) that $i_1=j_1$,\dots, $i_m=j_m$. Take an arbitrary point $x$ from the set $B_{i_1,\dots, i_n}$, the existence of which is guaranteed by the property~(C). Then $f_m(x)=y_{i_m,m}=g_m(x_0)$ and $f_n(x)=y_{i_n,n}=g_n(x_0)$.  Since the sequence $(f_n)_{n=1}^\infty$ is uniformly convergent on $E$, $\lim\limits_{n,m\to\infty} d(y_{i_m,m},y_{i_n,n})=0$. Hence, the sequence $(g_n)_{n=1}^\infty$ is uniformly convergent on $X$.

Since $Y$ is a complete space, for all $x\in X$ define $g(x)=\lim\limits_{n\to\infty}g_n(x)$. According to the property~(A), $g(x)=f(x)$ for all
  $x\in E$. Moreover, $g\in K_\alpha(X,Y)$ as a uniform limit of functions from the class $K_\alpha$.
\end{proof}

\section{Open problems}

\begin{question}
Does there exist a completely regular not perfectly normal space in which any functionally $G_\delta$-set is $1$-embedded?
\end{question}

\begin{question}
Does there exist a completely regular not perfectly normal space in which any set is $1$-embedded?
\end{question}

\begin{question}
Do there exist a normal space and its functionally $G_\delta$-subset which is not $1$-embedded?
\end{question}

\begin{question}
 Do there exist a topological space $X$ and its subspace $E$ such that $E$ is $K_1^*$-embedded and is not $K_1$-embedded in $X$?
\end{question}

\section{Acknowledgments} The author would like to thank the referee for his helpful and constructive comments that greatly contributed to improving the final version of the paper.

\end{document}